\documentclass[11pt,reqno]{amsart}
\usepackage[T1]{fontenc}
\usepackage{graphicx}
\usepackage{cite,enumerate}
\usepackage{multicol}

\usepackage{color}
\definecolor{MyLinkColor}{rgb}{0,0,0.4}

\newcommand{\id}{\mathop{\rm id}\nolimits}

\newcommand{\im}{\mathop{\rm Im}\nolimits}

\newcommand{\tr}{\mathop{\rm tr}\nolimits}
\newcommand{\Kern}{\mathop{\rm Ker}\nolimits}

\newcommand{\spa}{\mathop{\rm span}\nolimits}
\newcommand{\sign}{\mathop{\rm sign}\nolimits}

\newcommand{\0}{\Omega}
\newcommand{\e}{\varepsilon}

\newcommand{\p}{\partial}
\newcommand{\w}{\widetilde}
\newcommand{\ov}{\overline}
\newcommand{\wh}{\widehat}
\newcommand{\G}{\Gamma}

\newcommand{\A}{\mathcal{A}}
\newcommand{\B}{\mathcal{B}}

\newcommand{\cO}{\mathcal{O}}

\newcommand{\cF}{\mathcal{F}}
\newcommand{\cG}{\mathcal{G}}

\newcommand{\kL}{\mathcal{L}}

\newcommand{\kS}{\mathcal{S}}
\newcommand{\T}{\mathcal{T}}
\newcommand{\U}{\mathcal{U}}
\newcommand{\V}{\mathcal{V}}

\newcommand{\W}{\mathcal{W}}

\newcommand{\R}{\mathbb{R}}
\newcommand{\s}{\mathbb S}
\newcommand{\N}{\mathbb{N}}
\newcommand{\Z}{\mathbb{Z}}

\newtheorem{thm}{Theorem}[section]

\newtheorem{lemma}[thm]{Lemma}

\theoremstyle{remark} 
\newtheorem{rem}[thm]{Remark}

\numberwithin{equation}{section}

\title[Generalised Rayleigh-Taylor condition]{A generalised Rayleigh-Taylor condition for the Muskat problem}

\author[Escher]{Joachim Escher}
\address{Institut f{\"u}r Angewandte Mathematik, Leibniz Universit{\"a}t Hannover, Welfengarten~1, 30167 Hannover, Germany.}
\email{escher@ifam.uni-hannover.de}

\author[Matioc]{Anca-Voichita Matioc}
\address{Institut f{\"u}r Angewandte Mathematik, Leibniz Universit{\"a}t Hannover, Welfengarten~1, 30167 Hannover, Germany.}
\email{matioca@ifam.uni-hannover.de}

\author[Matioc]{Bogdan--Vasile Matioc}
\address{Institut f{\"u}r Angewandte Mathematik, Leibniz Universit{\"a}t Hannover, Welfengarten~1, 30167 Hannover, Germany.}
\email{matioc@ifam.uni-hannover.de}

\usepackage[colorlinks=true,linkcolor=MyLinkColor,citecolor=MyLinkColor]{hyperref} 

\begin{document}

\begin{abstract}
In this paper we consider the evolution of two fluid phases in a porous medium.
The fluids are separated from each other and also the  wetting phase from air by  interfaces which evolve in time.
We reduce the problem to an abstract evolution equation. A generalised Rayleigh-Taylor condition characterizes the parabolicity regime of the problem and allows us to  establish  a general well-posedness result and to study  stability properties of  flat steady-states.  
When considering surface tension effects at the interface between the fluids and if the more dense fluid lies above, we find bifurcating finger-shaped equilibria which are all unstable.
\end{abstract}

\maketitle
\section{Introduction}

The Muskat problem  is a widely used model for the 
intrusion of water into oil sand.
A linear analysis was performed in \cite{R, ST, T} where a relation, the so-called Rayleigh-Taylor condition, was found
to determine two regimes for the problem: a stable regime,  when a flat interface is stable under small deviations,  and an unstable one, when fingering occurs.

Nonetheless, existence and uniqueness of classical solutions has been  firstly  proven
in \cite{Y1} by using Newton's iteration method.
In the last decade the problem has received more interest and was studied by means of complex analysis \cite{SCH},
energy estimates \cite{Ambr, CCG1, CCG2, CG},  power series expansions \cite{FT}, or abstract parabolic theory \cite{EM12}.
These different approaches cover a wide spectrum of questions related to the Muskat problem: local well-posedness, global existence of solution, singular solutions, stability properties
of equilibria.  

It is worth noticing that all these papers mentioned above consider the situation when there is only one moving boundary, namely the one separating the fluids.
Either one prescribes boundary conditions at two boundaries which are kept fixed during the flow or so-called far-field boundary condition are imposed.
This setting  corresponds to an abstract equation with  only one unknown - the  interface between the fluids.
In the present paper we consider the more involved situation when there are two moving boundaries,  one separating the two fluids and one separating the wetting phase from air (assumed to be at uniform pressure equal to zero). 
The fluids are located in a porous medium (or a vertical Hele-Shaw cell) and  are assume to fill together with the dry phase (air) the entire void medium.
Moreover, we incorporate gravity and  viscosity  effects into the modeling as well as surface tension  forces at both interfaces.
The invertibility of a bounded  operator  permits us to  re-write the problem as an abstract  non-autonomous evolution equation
\[
\p_tZ=\Phi(t,Z),\qquad Z(0)=Z_0,
\]
where the variable $Z$  parametrises both  unknown interfaces.
The temporal variable $t$ is induced into the problem by the boundary condition $b$ for the pressure on the bottom of the cell.
For this problem we find  a generalised Rayleigh-Taylor condition in terms only  of the boundary data  $b$, the viscosities $\mu_\pm$, and densities $\rho_\pm$ of the fluids of the following form
\begin{equation}\label{eq:RayTay}
b\mu_++g\rho_+\mu_->0\qquad \text{and} \qquad\frac{\mu_+-\mu_-}{\mu_++\mu_-}(b-g\rho_+)+g(\rho_+-\rho_-)<0,
\end{equation}
which determines the parabolic character of the problem in the absence of surface tension effects.
When including surface tension forces at both interfaces we may drop condition \eqref{eq:RayTay}.  
We steadily use in this paper the subscript $-$ for the fluid on the bottom of the cell and $+$ for that above.
After showing that the Fr\'echet derivative $\p_Z\Phi(0)$ generates a strongly continuous and analytic semigroup, parabolic theory provides 
local well-posedness  of the problem and the principle of linearised stability may be applied to study the stability properties of the unique flat equilibrium which is determined for 
a fixed amount of fluid $+$ (this quantity is preserved by the flow) and a certain constant boundary data.

When considering surface tension effects at the interface between the fluids and the more dense fluid lies above we re-discover the  
global bifurcation branches obtained in \cite{EEM3, EM12}
which consist only of  finger-shaped equilibria of the Muskat problem.
The exchange of stability theorem due to Crandall and Rabinowitz \cite{CR3} applies to this particular problem and we show that all small 
equilibria are unstable.

The outline of the paper is as follows: we describe in Section 2 the mathematical model and present the main  results. 
Section 3 is dedicated to the proof of the well-posedness result Theorem \ref{T:P1}, and in the subsequent section we 
analyse the stability properties of the unique flat equilibrium as stated in Theorem \ref{T:P2}.
In Section 5 we prove our third main result, Theorem \ref{T:P3}.
The  calculations leading to the representation of $\p_Z\Phi(0)$ as a Fourier multiplication operator are done in the Appendix.

\section{The mathematical model and the main results}
Let us start this section by presenting the mathematical model of the setting described in the introduction.
Given $m\in\N$ and $\beta\in(0,1)$ the small H\"older space $h^{m+\beta}(\s)$ stands for the closure of the smooth functions $C^{\infty}(\s)$ in $C^{m+\beta}(\s).$
We let $\s$ denote the unit circle and functions on $\s$ are identified with $2\pi$-periodic functions on $\R.$
For later purposes we define  $h^{m+\beta}_{e}(\s)$ as the subspace of $h^{m+\beta}(\s)$ consisting only of even functions, 
 $h^{m+\beta}_{0}(\s)$ is the subspace of $h^{m+\beta}(\s)$ consisting only of functions with integral mean zero, and 
 $h^{m+\beta}_{0,e}(\s):= h^{m+\beta}_{0}(\s)\cap h^{m+\beta}_{e}(\s).$
Furthermore, we define the set of admissible functions to be
\[
\U:=\{f\in C^2(\s)\,:\, |f|<1/2\},
\]
Each pair  $(f,h)\in \U^2$ determines two open and simply connected subsets of the porous medium, seen as $\s\times(-1,2)\subset\s\times\R$, as follows:
\begin{align*}
&\Omega(f):=\{(x,y)\,:\, -1<y<f(x)\},\\[1ex]
&\Omega(f,h):=\{(x,y)\,:\, f(x)<y<1+h(x)\}.
\end{align*}
Let $T>0$ and  $(f,h) :[0,T]\to \U^2$ be given such that,  at each time $t\in[0,T],$  
the fluid $-$ is located at $\0(f(t))$ and  the fluid $+$  at $\0(f(t),h(t))$ (see Figure \ref{F:1}). 
\begin{figure}
\includegraphics[clip=true, angle=0,  trim = 0 -10 0 0, width=0.55\linewidth]{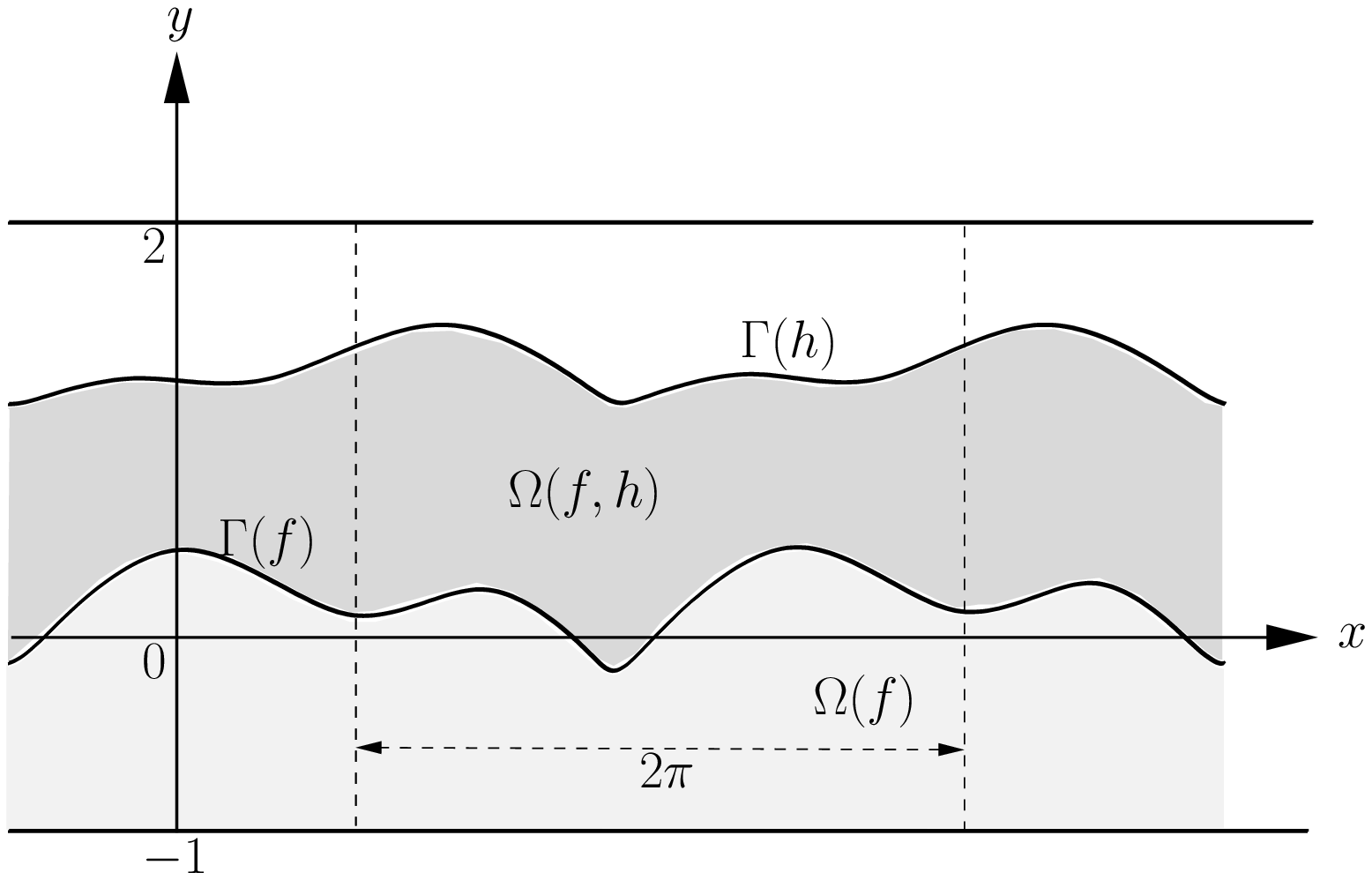}
\caption{The mathematical setting}
\label{F:1}
\end{figure}
The two fluids are assumed to be of Newtonian type and incompressible, and both interfaces  are supposed to move along with the fluids.
The problem is governed by the 
following system of partial differential equations:
\begin{equation}\label{eq:S}
\left\{\begin{array}{rllllll}
\Delta u_+&=&0&\text{in}& \Omega(f,h), \\[1ex]
\Delta u_-&=&0&\text{in}& \Omega(f), \\[1ex]
{\p_th+\frac{k\sqrt{1+h'^2}}{\mu_+}\p_\nu u_+}&=&0 &\text{on}& \Gamma(h),\\[1ex]
u_+&=&g\rho_+(1+ h)-\gamma_{d}\kappa_{\G(h)}&\text{on}&\Gamma(h),\\[1ex]
u_-&=&b&\text{on}&\Gamma_{-1},\\[1ex]
u_+-u_-&=&g(\rho_+-\rho_-)f+\gamma_{w}\kappa_{\G(f)}&\text{on}& \Gamma(f),\\[1ex]
{\p_tf+\frac{k\sqrt{1+f'^2}}{\mu_\pm}\p_\nu u_\pm}&=&0 &\text{on}& \Gamma(f),\\[1ex]
f(0)&=&f_0, \\[1ex]
h(0)&=&h_0
\end{array}
\right.
\end{equation} 
for $t\in[0,T]$, where $(f_0,h_0)\in\U^2$ determines the initial domains occupied by the fluids.
We used the variable $f$ for parametrising
the interface $\Gamma(f):=[y=f(x)]$ between the two fluids and $\Gamma(h):=[y=1+h(x)]$ separates the fluid $+$ from air.
The unit normal $\nu$ at $\Gamma(f)$ [resp. $\G(h)$] is chosen such that, if $\tau $ is the tangent,  the orthonormal basis $\{\tau,\nu\}$ has positive orientation. 
We also write $\kappa_{\G(f)}$ and $\kappa_{\G(h)}$ for the curvature of $\G(f)$ and $\G(h)$, respectively.
Moreover $\gamma_d $ [resp. $\gamma_w$] is the surface tension coefficient of the interface separating the fluids from air [resp. the fluids]. 
 
The potentials $u_\pm$ incorporate both pressure and gravity force $u_\pm:=p_\pm +g\rho_\pm y$, with $g$ the gravity constant.
 The velocity fields $\vec v_\pm$,  which satisfy Darcy's law 
\[
\vec v_\pm=-\frac{k}{\mu_\pm}\nabla u_\pm,
\]
are presupposed to be equal on the boundary separating the fluid phases.
Hereby,  $k$ stands for  the permeability of the porous medium.
On the fixed  boundary $\Gamma_{-1}:=\s\times\{-1\}$  we prescribed the value of the velocity potential $u_-$.
For a precise deduction of \eqref{eq:S} we refer to \cite{EM1, EM12, Y1}.

Let $\alpha\in(0,1)$ be fixed for the following.
A pair $(f,h, u_+,u_-)$ is called {\em classical H\"older solution} of \eqref{eq:S} if
\begin{align*}
&(f, h)\in C([0,T],\V)\cap C^{1}([0,T],(h^{1+\alpha}(\s))^2),\\[1ex]
&\text{$u_+(t)\in\mbox{\it buc}^{2+\alpha}(\0(f(t),h(t)))$ and $u_-(t)\in\mbox{\it buc}^{2+\alpha}(\0(f(t)))$ for $ t\in[0,T],$}
\end{align*}
and if $(f,h, u_+,u_-)$ satisfies the equations of \eqref{eq:S} pointwise. 
We defined  $\V:=\V_1\times\V_2$ to be  the subset of $\U^2$ given by
\begin{align*}
\V_1&:=\{f\in h^{2+2\sign(\gamma_w)+\alpha}(\s)\,:\, f\in\U\},\\
\V_2&:=\{h\in h^{2+2\sign(\gamma_d)+\alpha}(\s)\,:\, h\in\U\},
\end{align*}
where $\sign(0)=0$ and $\sign(\gamma)=1$ for $\gamma>0.$
The space $\mbox{\it buc}^{2+\alpha}(\0(f))$  is defined as closure of the smooth functions with bounded and  uniformly continuous derivatives 
$\mbox{\it BUC}\,^{ \infty}(\0(f))$ in $\mbox{\it BUC}\,^{ 2+\alpha}(\0(f))$.
The space   $\mbox{\it buc}^{2+\alpha}(\0(f,h))$ is defined similarly.
Moreover, since the potentials $u_\pm$ are determined,  when knowing  $(f,h)$, as solutions of elliptic problems (see Section 3) we   also refer to   $(f,h)$ to be the solution of \eqref{eq:S}.
The first main result of this paper states:

\begin{thm}\label{T:P1} Let $(\gamma_{d},\gamma_{w})\in(0,\infty)^2$  be given. 

There exist an open neighbourhood  $\cO$ of the zero function in  $\left(h^{4+\alpha}(\s)\right)^2$ 
such that for all $(f_0,h_0)\in\cO$ and $b\in C([0,\infty),h^{2+\alpha}(\s) )$ problem 
\eqref{eq:S} possesses a unique classical H\"older solution $\cF(\cdot;(f_0,h_0))$ defined  on a maximal time interval $[0,T(f_0,h_0))$  and which satisfies $\cF([0,T(f_0,h_0));(f_0,h_0))\subset\cO.$
The mapping $\{(t,f_0,h_0)\,:\,\text{$(f_0,h_0)\in\cO, t\in(0,T(f_0,h_0))$}\} \to\left(h^{4+\alpha}(\s)\right)^2$
\[
(t,f_0,h_0)\mapsto \cF(t;(f_0,h_0))
\]
has the same regularity as $b$ has.
\end{thm}

\begin{rem}\label{R:1} The conclusion of Theorem \ref{T:P1} remains valid if $\gamma_{d}=0$ or $\gamma_{w}=0$ with the following modifications:
 if $\gamma_{d}=\gamma_{w}=0$ we have to replace $\left(h^{4+\alpha}(\s)\right)^2$ by $\left(h^{2+\alpha}(\s)\right)^2$ and require that
$b\in C([0,\infty), c+\cO_0) $ where $\cO_0$ is a small neighbourhood of the zero function in $h^{2+\alpha}(\s)$ and $c\in\R$ satisfies
\begin{align}\label{eq:cond1}
&c\mu_++g\rho_+\mu_->0,\\[1ex]
\label{eq:cond2}
&\frac{\mu_+-\mu_-}{\mu_++\mu_-}(c-g\rho_+)+g(\rho_+-\rho_-)<0.
\end{align}
When $(\gamma_{d},\gamma_{w})\in \{0\}\times (0,\infty)$ [resp. $(\gamma_{d},\gamma_{w})\in  (0,\infty)\times \{0\}$]
 we replace $\left(h^{4+\alpha}(\s)\right)^2$ by $h^{4+\alpha}(\s)\times h^{2+\alpha}(\s)$
[resp. $h^{2+\alpha}(\s)\times h^{4+\alpha}(\s)$] and request that the constant $c$ satisfies only equation \eqref{eq:cond1} [resp. eq. \eqref{eq:cond2}].
\end{rem}

Relation  \eqref{eq:cond1} is  a generalisation of the positive pressure  condition imposed   in \cite{EM1, EM7, EM5} 
to ensure well-posedness and stability of the one-phase Hele-Shaw problem without surface tension. 
Indeed, if the fluids have the same densities and viscosity, \eqref{eq:cond1} re-writes $c+g\rho_+>0$ which is, up to a scaling, the same condition as in \cite{EM1, EM7, EM5}. 
Moreover, it turns out that the Muskat problem without surface tension effects studied in \cite{CG, EM12, Yi2} is similar to our problem if  $\gamma_d>0.$ 
Indeed, we have:
\begin{lemma}\label{L:1} The volume of fluid $+$ is preserved by the solutions of \eqref{eq:S}. 
\end{lemma}
\begin{proof} The proof is similar to that of \cite[Lemma 3.1]{EM7}.
\end{proof}

In order to establish similarity between our problem when $\gamma_d>0$ and that in \cite{EM12, Yi2}, we determine a
special  solution of \eqref{eq:S} in the case when the volume of fluid $+$ is equal to $2\pi$, i.e.
\begin{equation}\label{eq:cond3}
\int_{\s}f_0-h_0\, dx=0.
\end{equation}
If  initially $f(0)=h(0)=f_0\in\R$  and $b$ depends only on time, 
 then 
\begin{equation}\label{eq:pro-3}
f'(t)=-\frac{kg\rho_-}{\mu_-}\frac{ f(t)+\frac{g\rho_+-b}{g\rho_-}}{f(t)+\frac{\mu_++\mu_-}{\mu_-}},\qquad f(0)=f_0,
\end{equation}
and, by Lemma \ref{L:KL}, $f(t)=h(t)$ as long as the solution exists.
If $\rho_+=\rho_-$ and $\mu_->\mu_+$ we obtain from \eqref{eq:cond2} that  if $\gamma_w=0$, then $b>g\rho_+$, 
thus $f'$ is positive if $f_0$ is close to zero, meaning that the more viscous fluid 
drives upwards the less viscous one in the medium.  
This  condition has been found also in \cite{EM12, Yi2} to guarantee well-posedness of the Muskat problem  studied therein.
Moreover, if the Atwood number 
\[
A_\mu:=(\mu_+-\mu_-)/(\mu_++\mu_-)
\]
 is zero,  then \eqref{eq:cond2} tells us that the more dense fluid must lay beneath in order  to guarantee well-posedness of \eqref{eq:S} when $\gamma_w=0,$
result  similar to that in \cite{CG, EM12}. 

Corresponding to the result in \cite{EM12}, where an optimal value for the normal
 velocity at which water may replace oil in the absence of surface tension effects was found,
 we obtain herein an optimal value for the pressure on the bottom of the medium:  
\begin{rem}
If the fluid below is water and that above oil, and we neglect the surface force at the interface between them, we find from \eqref{eq:cond2} an optimal value
\begin{equation}\label{eq:optimal}
p_{max}:=g(\rho_++\rho_-)-g(\rho_+-\rho_-)A_\mu^{-1}
\end{equation}
for the pressure on the bottom of the porous medium below which   water may drive upwards  oil in a stable regime (no fingering occurs).
\end{rem}
\begin{proof} Relation \eqref{eq:optimal} is obtained form \eqref{eq:cond2} in view of $u_-=p_--g\rho_-$ on $\Gamma_{-1}.$ 
The optimal value for the potential  $b$ is $b_{max}=p_{max}-g\rho_-,$ and if the boundary value $b$ is close to this value we find that the solutions of \eqref{eq:pro-3}
fulfill $f'>0,$ thus water drives  oil upwards.
This last assertion follows  from
\[
g\rho_+-b_{max}=g(\rho_+-\rho_-)A_\mu^{-1}<0
\]
since it is well-known \cite{CC} that $\rho_+<\rho_-$ and $\mu_+>\mu_-$ (oil is  less dense and more viscous than water). 
\end{proof}

We infer  from \eqref{eq:pro-3} that if $b= g\rho_+$ and $f_0=0$, then $f(t)=h(t)=0$ for all $t\geq0.$
Concerning the stability properties of the stationary solution  $(f,h)=(0,0)$, which is the unique flat stationary solution of \eqref{eq:S} for $b=g\rho_+$ and which satisfies \eqref{eq:cond3},
we state:
\begin{thm}\label{T:P2} Let $\gamma_d, \gamma_w\in[0,\infty)$. 
Then:
\begin{itemize}
\item[$(i)$] If $g(\rho_--\rho_+)+\gamma_w>0,$ then the flat equilibrium $(f,h)=(0,0)$ of \eqref{eq:S} is exponentially stable.
More precisely, there exists positive constants $M, \delta,$ and $\omega$ such that if $\|(f_0,h_0)\|_{h^{2+2\sign(\gamma_w)+\alpha}(\s)\times h^{2+2\sign(\gamma_d)+\alpha}(\s)}\leq\delta$
and $(f_0,h_0)$ satisfies  \eqref{eq:cond3}, then the solution $(f,h)$ of \eqref{eq:S} exists globally and 
\begin{align*}
&\|(f(t),h(t))\|_{h^{2+2\sign(\gamma_w)+\alpha}(\s)\times h^{2+2\sign(\gamma_d)+\alpha}(\s)}+\|(\p_tf(t),\p_th(t))\|_{(h^{1+\alpha}(\s))^2}\\[1ex]
&\phantom{doublespace}\leq Me^{-\omega t}\|(f_0,h_0)\|_{h^{2+2\sign(\gamma_w)+\alpha}(\s)\times h^{2+2\sign(\gamma_d)+\alpha}(\s)},\quad\forall t\geq0.
\end{align*}

\item[$(ii)$] If $g(\rho_--\rho_+)+\gamma_w<0,$ then $(f,h)=(0,0)$  is unstable.
\end{itemize}
\end{thm}

\begin{rem}\label{R:2} When we  study the  stability of equilibria  in Theorem \ref{T:P2} and Theorem \ref{T:P3} below we  fixed a volume of fluid $+$ equal to $2\pi,$
meaning that the initial data of \eqref{eq:S} are presupposed to satisfy \eqref{eq:cond3}.
This setting is imposed by Lemma \ref{L:1}, since the volume of fluid $+$ is preserved   by the solutions of \eqref{eq:S}.
\end{rem}
Theorem \ref{T:P2} is related to the exponential stability result established in \cite[Theorem 5.3]{EM12}
for the Muskat problem with only one free boundary and is  stronger than that in \cite{FT}, where only stability is shown.
Notice that if $\gamma_w=0,$ then the flat solution is always stable, since  $\rho_->\rho_+$  is exactly the condition \eqref{eq:cond2} which guarantees well-posedness of \eqref{eq:S}.
Concerning the unstable case, numerical experiments \cite{HLS} show that the interface between the fluids becomes very  ramified, and dendrite like structures occur as time evolves if $g(\rho_--\rho_+)+\gamma_w<0.$ 

If $b=g\rho_+$ and the volume of fluid $+$ is equal to $2\pi$, there exist also other stationary solutions of \eqref{eq:S}.
They appear only in the unstable regime or sufficiently close  to it, that is when $\gamma_w>0$ and the more dense fluid lies above in the cell.
We show that for certain small $\gamma_w>0$ there exist finger-shaped stationary solutions of \eqref{eq:S}, and therefore we shall refer also to $(\gamma_w,f,h)$ to be solution of \eqref{eq:S}.
Given $1\leq l\in\N, $ we define
\[
\ov\gamma_l:=\frac{g(\rho_+-\rho_-)}{l^2}.
\]
\begin{figure}
\includegraphics[clip=true, angle=0,  trim = 0 -10 0 0, width=0.55\linewidth]{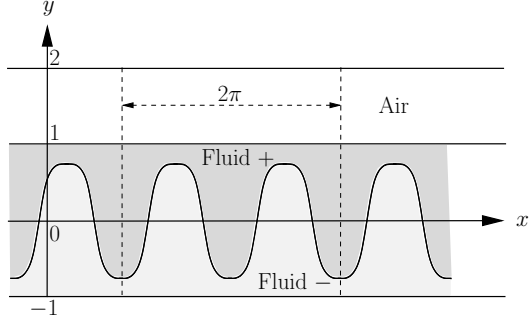}
\caption{Steady-state solution of \eqref{eq:S} when $\rho_-<\rho_+$ and $\gamma_w>0.$}
\label{F:2}
\end{figure}
\begin{thm}\label{T:P3} Let $b=g\rho_+$ and $\gamma_w(\rho_+-\rho_-)>0$.
If  $(\gamma_w,f,h)$  is a stationary solution of \eqref{eq:S} satisfying \eqref{eq:cond3}, then $h=0$ and  $(\gamma_w,f)$ is a solution 
 of the Laplace-Young equation
\begin{align}\label{eq:FiF1}
\gamma_w\kappa(f)+g(\rho_+-\rho_-)f=0. 
\end{align}
The solution of \eqref{eq:FiF1} are, up to a translation, even 
and all even solutions of \eqref{eq:FiF1} can be represented as a disjoint union
\[ 
\cup_{l=1}^\infty\{(\gamma_l(\e), f_l(\e))\,:\, \e\in\R\}\bigcup\cup_{l=1}^\infty(\ov\gamma_{l+1},\ov\gamma_l)\times\{0\}\bigcup (\ov\gamma_1,\infty)\times\{0\},
\]
with continuous functions
\[
(\gamma_l,f_l):\R\to (0,\infty)\times \{f\in h^{4+\alpha}_{0,e}(\s)\,:\,\|f\|_{C(\s)}<1\},\qquad 1\leq l\in\N, 
\]
 which, near $\e=0$, are real analytic  and satisfy:
\begin{align*} 
\gamma_l(\e)&=\ov\gamma_l +\frac{3 g(\rho_+-\rho_-)}{8}\e^2+O(\e^4),\qquad f_l(\e)&=\e\cos(lx)+O(\e^2). 
 \end{align*}
While $\gamma_l$ is even and  
\[
\lim_{|\e|\nearrow \infty}\gamma_l(\e)\leq \frac{2\pi^2g(\rho_+-\rho_-)}{B^2(3/4,1/2)l^2},
\]
either $\|f_l(\e)\|_{C(\s)}\nearrow_{\e\to\infty} 1$ or $\|f_l'(\e)\|_{C(\s)}\nearrow_{\e\to\infty} \infty.$

Additionally, the equilibrium $(\gamma_l(\e),f_l(\e),0)$ of problem \eqref{eq:S} is unstable if $|\e| $ is small.
When $l=1$ we have to assume   $\e\neq0$ too.
\end{thm}

Here $B$ stands for Euler's beta function.
Notice that the stationary solutions of \eqref{eq:S}, which satisfy \eqref{eq:cond3} (see in Figure \ref{F:2}), 
are the same with the stationary solutions of the Muskat problem studied in \cite{EM12, CG}, 
where   just one moving boundary is considered ($h$ is chosen a priori to be zero).
For a precise description of the global bifurcation branches $(\gamma_l,f_l)$ we refer to \cite{EEM3}.
It is shown there that the situation  $\|f_l(\e)\|_{C(\s)}\nearrow_{\e\to\infty} 1$ may occur only for small integers $l$.

\section{The evolution  equation}
In order to solve problem \eqref{eq:S} we re-write it as an abstract evolution equation on the unit circle.
To do that we first transform system \eqref{eq:S} into a system of equations on fixed domains by using the unknown functions $(f,h).$
Let $\0_-:=\0(0)$ and $\0_+:=\0(0,0).$
Given $(f, h)\in\V$ we define the mappings
$\phi_{ f}=(\phi_{f}^1,\phi_{f}^2):\Omega_{-}\to\Omega(f)$
by 
\[
\phi_{f}(x,y):=(x,y+(1+ y)f(x)),\quad(x,y)\in\Omega_-,
\]
respectively $\phi_{ f,h}=(\phi_{f,h}^1,\phi_{f,h}^2):\Omega_{+}\to\Omega(f,h)$
\[
\phi_{f,h}(x,y):=(x, y (1+h(x))+(1-y)f(x)),\quad(x,y)\in\Omega_+.
\]
One can easily check that $\phi_f$ and $\phi_{f,h}$ are diffeomorphisms for all $(f, h)\in\V.$ 
These diffeomorphisms induce pull-back and push-forward operators (see e.g. \cite{EM1}) which we use to transform the differential operators involved in system \eqref{eq:S}
into operators on the domains $\0_\pm$ and their boundaries, respectively. 
Each pair  $(f, h)\in\V$ induces linear elliptic operators
\begin{eqnarray*}
&\A(f):\mbox{\it{buc}}\,^{2+\alpha}(\Omega_-)\to \mbox{\it{buc}}\,^{\alpha}(\Omega_-),& \A(f)v:=\Delta(v_-\circ\phi_{f}^{-1})\circ\phi_{f},\\[1ex]
&\A(f,h):\mbox{\it{buc}}\,^{2+\alpha}(\Omega_+)\to \mbox{\it{buc}}\,^{\alpha}(\Omega_+),& \A(f,h)v:=\Delta(v_+\circ\phi_{f,h}^{-1})\circ\phi_{f,h},
\end{eqnarray*}
which depend, as bounded operators, analytically on $f$ and $h$.
Denote by $\tr_0$ the trace operator with respect to $\Gamma_0:=\s\times\{0\}$.
We associate problem \eqref{eq:S} the following   trace operators  on $\Gamma_0$:
\begin{align*}
B(f)v_-&:=k\mu_-^{-1}\tr_0 (\langle \nabla(v_-\circ\phi_{f}^{-1})|(-f',1)\rangle\circ\phi_{f}),\quad v_-\in \mbox{\it{buc}}\,^{2+\alpha}(\Omega_-),\\[1ex]
\B(f,h)v_+&:=k\mu_+^{-1}\tr_0 (\langle \nabla(v_+\circ\phi_{f,h}^{-1})|(-f',1)\rangle\circ\phi_{f,h}),\quad v_+\in \mbox{\it{buc}}\,^{2+\alpha}(\Omega_+),
\end{align*}
which, seen as bounded operators into $h^{1+\alpha}(\s)$,  depend analytically on $f$ and $h$ as well.
Lastly, we define  a boundary operator on $\Gamma_1:=\s\times\{1\}.$
Given $(f,h)\in\V,$ we let
\[
\B_1(f,h)v_+:=-k\mu_+^{-1}\tr_1 (\langle \nabla(v_+\circ\phi_{f,h}^{-1})|(-h',1)\rangle\circ \phi_{f,h}),\quad v_+\in \mbox{\it{buc}}\,^{2+\alpha}(\Omega_+),
\]
whereby  $\tr_1$ is the trace operator with respect to $\Gamma_1$.

With this notation one can easily verify that if 
 $(f,h,u_+,u_-)$ is a solution of \eqref{eq:S},
then $(f,h, v_+:=u_+\circ\phi_{f,h}, v_-:=u_-\circ\phi_f)$
solves the following system of equations:
\begin{equation}\label{eq:TS}
\left\{
\begin{array}{rllllll}
\A(f,h) v_+&=&0&\text{in}& \Omega_+,\\[1ex]
\A(f) v_-&=&0&\text{in}& \Omega_-,\\[1ex]
\p_t h&=&\B_1(f,h)v_+&\text{in}& \Gamma_1,\\[1ex]
v_+&=&g\rho_+(1+h)-\gamma_{d}\kappa(h)&\text{on}& \Gamma_1,\\[1ex]
 v_-&=&b&\text{on}&\Gamma_{-1},\\[1ex]
v_+-v_-&=&g(\rho_+-\rho_-) f+\gamma_{w} \kappa(f)&\text{on}& \Gamma_0,\\[1ex]
\p_t f+\B(f)v_-&=&0 &\text{on}& \Gamma_{0},\\[1ex]
\p_t f+\B(f,h)v_+&=&0 &\text{on}& \Gamma_{0},\\[1ex]
f(0)&=&f_0,&&\\[1ex]
h(0)&=&h_0&&
\end{array}
\right.
\end{equation}
for all $t\in[0,T]$, where 
the transformed curvature operator $\kappa:h^{4+\alpha}(\s)\to h^{2+\alpha}(\s)$  is defined by $\kappa(f):=f''/(1+f'^2)^{3/2}.$
The notion of solution of \eqref{eq:TS} is defined analogously to that of \eqref{eq:S}.
Notice that the parametrisation $(f,h)$ is left invariant by the transformation above. 
In fact, one can see, cf. \cite[Lemma 1.2]{EM1} that each solution of \eqref{eq:TS} corresponds to a unique solution of \eqref{eq:S}.

We introduce now solution operators corresponding to the system \eqref{eq:TS}.
Given $f\in \V_1$ and $(q,p)\in h^{1+\alpha}(\s)\times h^{2+\alpha}(\s)$,  
we let $\T(f,q,p)\in \mbox{\it{buc}}\,^{2+\alpha}(\Omega_-)$ denote the solution of the linear, elliptic mixed boundary value problem
\begin{equation}\label{eq12}\left\{
\begin{array}{rlllll}
\A(f) v_-&=&0&\text{in}& \Omega_-,\\[1ex]
\B(f)v_-&=&q &\text{on}& \Gamma_{0},\\[1ex]
v_-&=&p&\text{on}& \Gamma_{-1}.
\end{array}
\right.
\end{equation}
Further on, we define  
$\kS:\V\times (h^{2+\alpha}(\s))^2\to\mbox{\it{buc}}\,^{2+\alpha}(\Omega_+)$
 by writing 
$\kS(f,h,p,r)$ for the unique solution of the problem
\begin{equation}\label{eq13}\left\{
\begin{array}{rlllll}
\A(f,h) v_+&=&0&\text{in}& \Omega_+,\\[1ex]
v_+&=&p&\text{on}& \Gamma_1,\\[1ex]
 v_+&=&r&\text{on}&\Gamma_{0}.
\end{array}
\right.
\end{equation}
It is convenient to write $\T(f,q,p)=\T_1(f)q+\T_2(f)p,$
where
\begin{align*}
\T_1(f)q:=(\A(f),\B(f),\tr)^{-1}(0,q,0),\quad \T_2(f)p:=(\A(f),\B(f),\tr)^{-1}(0,0,p),
\end{align*}
respectively $\kS(f,h,p,r)=\kS_1(f,h)p+\kS_2(f,h)r,$ with
\begin{align*}
\kS_1(f,h)p: &=(\A(f,h),\tr,\tr)^{-1}(0,p,0),\\[1ex] 
\kS_2(f,h)r:&=(\A(f,h),\tr,\tr)^{-1}(0,0,r).
\end{align*}
The operators $\T_i(f)$ and $\kS_i(f,h),$ $i=1,2,$ are bounded linear operators and they depend, in the norm topology, analytically on $f$ and $h$ too.

The key point of our analysis is the following observation.
If  $(f,h, v_+,v_-)$ is a classical solution of \eqref{eq:TS} for to the initial data $(f_0,h_0)$, then it must hold:
\begin{itemize}
\item[$(i)$] $f(0)=f_0$ and $h(0)=h_0,$
\item[$(ii)$] $v_-=\T(f,-\p_t f, b),$
\item[$(iii)$] $v_+=\kS(f,h,g\rho_+(1+h)-\gamma_{d}\kappa(h),\tr_0v_-+g(\rho_+-\rho_-) f+\gamma_{w}\kappa(f)),$
\item[$(iv)$] $\p_tf+\B(f,h)v_+=0,$
\item[$(v)$] $\p_th=\B_1(f,h)v_+.$
\end{itemize}

Let us now show that from $(ii)-(iv)$ we can determine the derivative $\p_tf$ as a function  of $f$, $h$, and $t$ only.
Indeed, we plug $(ii)$  into $(iii)$ and $(iii)$ into $(iv)$ to obtain the equation
\begin{align*}
&\p_tf+\B(f,h)\kS_1(f,h)[g\rho_+(1+h)-\gamma_d\kappa(h)]\\[1ex]
&\phantom{space=}+\B(f,h)\kS_2(f,h)[\tr_0\T(f,-\p_t f, b)+g(\rho_+-\rho_-) f+\gamma_{w}\kappa(f)]=0,
\end{align*}
 which can be writen equivalently
\begin{equation}\label{eq:in}
\begin{aligned}
&(\id_{h^{1+\alpha}(\s)}-\B(f,h)\kS_2(f,h)\tr_0\T_1(f))\p_t f
+\B(f,h)\kS_2(f,h)\tr_0\T_2(f)b\\[1ex]
 &\phantom{space=}+\B(f,h)\kS(f,h,g\rho_+(1+h)-\gamma_{d}\kappa(h),g(\rho_+-\rho_-) f+\gamma_{w}\kappa(f)).
\end{aligned}
\end{equation}
The linear operator which is evaluated at $\p_tf$ is invertible, so that we obtain, by applying its inverse to \eqref{eq:in},
an equation expressing the derivative $\p_tf$ in dependence of $f, h,$ and $t$.
Indeed, we have:
\begin{lemma}\label{L:KL} The set $\V$ contains an open neighbourhood $\W$ of $0$ with the property that
\[
\cG(f,h):=\id_{h^{1+\alpha}(\s)}-\B(f,h)\kS_2(f,h)\tr_0\T_1(f)\in\kL(h^{1+\alpha}(\s))
\] 
is an isomorphism for all $(f,h)\in\W.$
\end{lemma}
\begin{proof} 
The proof is based on a continuity argument. 
Namely, all the operators defined in this section depend analytically on their variables and then so does  $\cG$ too.
Thus, it suffices  to show that $\cG(0,0)$ is an isomorphism.
To do that, we represent  $\cG(0,0)$ as a Fourier multiplication operator.
Given $q\in h^{1+\alpha}(\s)$ we let $q=\sum_{m\in\Z}\wh h(m) e^{imx}$  denote its Fourier series expansion.
A Fourier series ansatz yields for 
 $\T_1(0)q$ the following   expression
\[
\T_1(0)q(x,y)=\frac{\mu_-}{k}(1+y)\wh q(0)
+\frac{\mu_-}{k}\sum_{m\in\Z\setminus\{0\}}\frac{e^{m}e^{my}-e^{-m}e^{-my}}{m(e^{m}+e^{-m})}\wh q(m) e^{imx}
\]
for $(x,y)\in\0_-.$
Respectively, if $r=\sum_{m\in\Z}\wh r(m) e^{imx},$ then $\kS_2(0,0)r$ may be  expanded as follows
\[
\kS_2(0,0)r(x,y)=(1-y)\wh r(0)+\sum_{m\in\Z\setminus\{0\}}\frac{e^{2m}e^{-my}-e^{my}}{e^{2m}-1}\wh r(m)e^{imx}
\]
 for $(x,y)\in\0_+.$
Combining these two relations  and taking the normal derivative yields that
\begin{equation*}
\cG(0,0)q=\frac{\mu_-+\mu_+}{\mu_+}q, \quad\forall q\in h^{1+\alpha}(\s),
\end{equation*}
thus $\cG(0,0)$ is an isomorphism.
\end{proof}
In virtue of Lemma \ref{L:KL},  if the pair $(f,h)$  maps into $\W,$ we may apply the inverse of $\cG(f,h)$ to  \eqref{eq:in}, 
and get
 \begin{equation}\label{eq:EV1}
\p_tf=\Phi_1(t,f,h),
\end{equation}
with a nonlinear and nonlocal operator $\Phi_1$ defined by the relation
\begin{equation}\label{eq:PHI1}
\begin{aligned}
\Phi_1(t,f,h):=&-\cG^{-1}(f,h)\left\{\B(f,h)\kS_2(f,h)\tr_0\T_2(f)b\right.\\[1ex]
&+\B(f,h)\kS(f,h,g\rho_+(1+h)-\gamma_{d}\kappa(h),g(\rho_+-\rho_-) f+\gamma_{w}\kappa(f)).
\end{aligned}
\end{equation}
Furthermore, from $(ii)-(v)$ and \eqref{eq:EV1} we  obtain that $h$ is solution of the equation
 \begin{equation}\label{eq:EV2}
\p_th=\Phi_2(t,f,h),
\end{equation}
where the operator $\Phi_2$  is given by
\begin{equation}\label{eq:PHI2}
\begin{aligned}
\Phi_2(t,f,h):=&B_1(f,h)\kS(f,h,g\rho_+(1+h)-\gamma_{d}\kappa(h),g(\rho_+-\rho_-) f+\gamma_{w}\kappa(f))\\[1ex]
&+B_1(f,h)\kS_2(f,h)\tr_0\T(f,-\Phi_1(t,f,h), b).
\end{aligned}
\end{equation}

By Lemma \ref{L:KL} and relations \eqref{eq:EV1}-\eqref{eq:PHI2} we found that all the solutions $(f,h)$ of \eqref{eq:TS} which are contained  in $\W$
solve the following abstract evolution equation
\begin{equation}\label{eq:ABS}
\p_tZ=\Phi(t,Z),\qquad Z(0)=Z_0, 
\end{equation}
where $\Phi:=(\Phi_1,\Phi_2)$ and we introduced the new   variable $Z:=(f,h)$.
Concerning the operator $\Phi$ we state:
\begin{thm}\label{T:PHI-O} The operator $\Phi$ has the same regularity as $b$, it is analytic in the variable $Z$, and
if $b(0)=:c\in\R,$ then   $\p_f\Phi_i(0)$ and $\p_h\Phi_i(0),$ $i=1,2,$ are Fourier multipliers with symbols $(\lambda^f_i(m))_{m\in\Z}$ and  $(\lambda^h_i(m))_{m\in\Z},$ $i=1,2,$ 
 respectively, given by 
 \begin{align}\label{eq:coef_1_f}
&\lambda^f_1(m):=\left[A_\mu(c-g\rho_+)+g(\rho_+-\rho_-)-\gamma_{w}m^2\right]\frac{k|m|}{(\mu_++\mu_-)\tanh(|m|)},\\[1ex]
\label{eq:coef_2_f}
&\lambda^f_2(m):=\left[\frac{c(\mu_+-\mu_-)+2g\rho_+\mu_-}{\mu_++\mu_-}-g\rho_--\gamma_wm^2\right]\frac{k|m|}{(\mu_++\mu_-)\sinh(|m|)},\\[1ex]
\label{eq:coef_1_h}
&\lambda^h_1(m):=-\left[\frac{g\rho_+\mu_-+c\mu_+}{\mu_-+\mu_+}+\gamma_{d}m^2\right]\frac{k|m|}{(\mu_-+\mu_+)\sinh(|m|)},
\end{align}
\begin{align}
\label{eq:coef_2_h}
&\lambda^h_2(m):=-\left[\frac{c\mu_++g\rho_+\mu_-}{\mu_-+\mu_+}+\gamma_dm^2\right]\frac{k|m|}{\mu_+\tanh(|m|)}-\frac{\mu_-}{\mu_+}\frac{\lambda^h_1(m)}{\cosh(m)}.
\end{align}
\end{thm}
\begin{proof} The  regularity assertion is obvious.
That the first order partial derivatives of $\Phi_i, i=1,2,$ with respect to $f$ and $h$ are Fourier multipliers follows from \eqref{eq:P1f}-\eqref{eq:P2h}, relations  
proven in the Appendix.
\end{proof}
 
 We give now a proof of our first main result.
 
 \begin{proof}[Proof of Theorem \ref{T:P1} ] 
 We  verify  that the assumptions of \cite[Theorem 8.4.1]{L} are fulfilled by $\Phi.$
 Theorem \ref{T:P1} is then a consequence of this result. 
 For continuity reasons it suffices in fact to show only that the derivative $\p_Z\Phi(0)$ generates a strongly continuous and analytic semigroup in $\kL((h^{1+\beta}(\s))^2),$
 i.e.
 \[
-\p_Z\Phi(0)\in\mathcal{H}(h^{2+2\sign(\gamma_w)+\beta}(\s)\times h^{2+2\sign(\gamma_d)+\beta}(\s), (h^{1+\beta}(\s))^2)
\]
for some $\beta\in(0,\alpha).$
By using the interpolation properties of the small H\" older spaces
\begin{equation}\label{interpolation}
(h^{\sigma_0}(\s),h^{\sigma_1}(\s))_\theta=h^{(1-\theta)\sigma_0+\theta\sigma_1}
(\s),
\end{equation}
if $  \theta\in(0,1)$ \text{and} $(1-\theta)\sigma_0+\theta\sigma_1\notin\N,$
 we find then all assumptions of \cite[Theorem 8.4.1]{L} 
to be fulfilled.
Here $(\cdot,\cdot)$ denotes the interpolation functor introduced by Da Prato and Grisvard \cite{DG}.

Let us first notice  that derivative $\p_h\Phi_1(0)$ maps $h^{2+2\sign(\gamma_d)+\w\beta}(\s)$ continuously into $h^{1+\beta}(\s)$ for some $\w\beta\in(0,\beta)$.
This property can be verified easily by using \cite[Theorem 3.4]{EM1}, which is a multiplier theorem based on some generalized Marcinkiewicz conditions.
 Since by \eqref{interpolation}
\[
h^{2+2\sign(\gamma_d)+\w\beta}(\s)=(h^{1+\beta}(\s), h^{2+2\sign(\gamma_d)+\beta}(\s))_{(1+2\sign(\gamma_d)+\w\beta-\beta)/(2\sign(\gamma_d)+1)}
\]
we deduce, in virtue of Theorem 1.6.1 and relation (2.2.2) in \cite{Am}, that $\p_Z\Phi(0)$ generates a strongly continuous and analytic semigroup exactly when 
$\p_f\Phi_1(0)$ and $\p_h\Phi_2(0)$ are generators, i.e.
\begin{align*}
&-\p_f\Phi_1(0)\in\mathcal{H}(h^{2+2\sign(\gamma_w)+\beta}(\s), h^{1+\beta}(\s)),\\[1ex]
&-\p_h\Phi_2(0)\in\mathcal{H}( h^{2+2\sign(\gamma_d)+\beta}(\s), h^{1+\beta}(\s)).
\end{align*}

When considering surface tension effects this property holds independently of the boundary conditions, and when
$\gamma_w=0$ and $\gamma_d=0$ this is true  if
\[
\frac{\mu_+-\mu_-}{\mu_++\mu_-}(c-g\rho_+)+g(\rho_+-\rho_-)<0\quad
\text{and}
\quad 
\frac{c\mu_++g\rho_+\mu_-}{\mu_-+\mu_+}>0,
\]
respectively, with  $b(0)$ sufficiently close to $c\in\R$ in $h^{2+\beta}(\s)$. 
We refer to \cite{EM1} where the generator property of a Fourier multiplier between  space of periodic and continuous functions is explicitly verified when knowing its symbol.
In virtue of \eqref{interpolation}  the proof is completed.
\end{proof}

\section{Equilibria and stability properties}

As we mentioned earlier, if we consider a fixed volume of fluid $+$ equal to $2\pi$ and if  $b=g\rho_+$, then $(f,h)=(0,0)$ is the unique flat  stationary solution of problem \eqref{eq:S}. 
Moreover, the reduced equation \eqref{eq:ABS} is autonomous since $\Phi$ does not depend on time for  constant $b$.
In order to study the stability properties of this equilibrium, as stated in Theorem  \ref{T:P2}, we shall use the principle of linearised stability, and need therefore to determine the spectrum of 
the derivative $\p\Phi(0).$
Being a  generator and taking into consideration that  the small H\"older space  $ h^{2+\alpha}(\s)$ is compactly embedded into $h^{1+\alpha}(\s)$,
 we obtain from  \cite[Theorem III.8.29]{K} that its spectrum consists entirely of isolated eigenvalues with finite multiplicity.

 In virtue of Theorems 9.1.2 and 9.1.3 in \cite{L} we know that the trivial solution $(f,h)=(0,0)$ is 
exponentially stable if the spectrum of $\p\Phi(0)$ is bounded away from the imaginary axis in the left half complex plane, 
and unstable if the infimum of the real part of all eigenvalues in the right half plane is positive. 
One can easily see that if $\lambda$ is an eigenvalue of $\p\Phi(0),$ then it must be, for some $m\in\N,$ eigenvalue of the matrix
\[
\left[
\begin{array}{cc}
\lambda_1^f(m)&\lambda_1^h(m)\\[1ex]
\lambda_2^f(m)&\lambda_2^h(m)
\end{array}
\right],
\]
where, for $b=g\rho_+$, we obtained the simpler expressions for the multiplier symbols:  
\begin{equation}\label{eq:symb}
\begin{aligned}
&\lambda^f_1(m)=-\left[g(\rho_--\rho_+)+\gamma_{w}m^2\right]\frac{km}{(\mu_++\mu_-)\tanh(m)},\\[1ex]
&\lambda^f_2(m)=-\left[g(\rho_--\rho_+)+\gamma_wm^2\right]\frac{km}{(\mu_++\mu_-)\sinh(m)},\\[1ex]
&\lambda^h_1(m)=-\left[g\rho_++\gamma_{d}m^2\right]\frac{km}{(\mu_-+\mu_+)\sinh(m)},\\[1ex]
&\lambda^h_2(m)=-\left[g\rho_++\gamma_dm^2\right]\frac{km}{\mu_++\mu_-}\frac{\mu_-(\cosh^2(m)-1)+\mu_+\cosh^2(m)}{\mu_+\sinh(m)\cosh(m)}.
\end{aligned}
\end{equation}
Thus, the spectrum of $\p\Phi(0)$ consists only of the eigenvalues
\begin{equation*}
\Lambda_\pm(m):=\frac{\lambda^f_1(m)+\lambda^h_2(m)\pm\sqrt{(\lambda^f_1(m)-\lambda^h _2(m))^2+4\lambda^h_1(m)\lambda_2^f(m)}}{2},
\end{equation*}
whereby $m\in\N.$
Easily, we see that $\Lambda_+(0)=0,$ thus we find our selves in the critical case of stability when $0$ is an eigenvalue, which makes it difficult for us
to establish the stability properties of the flat solutions.
This is due to the fact that the volume of fluid $+$ is preserved by the flow, and this property has not been included yet into our equations \eqref{eq:ABS}.
We do this by introducing a new variable $\w f:=f-h\in h^{2+2\sign(\gamma_d\cdot\gamma_w)+\alpha}_0(\s)$.
Then
\begin{equation}\label{eq:PSI}
\begin{aligned}
\p_t\w f&=\p_t f-\p_th=\Phi_1(f,f-\w f)-\Phi_2(f,f-\w f)=:\Psi_2(f,\w f), \\[1ex]
\p_t f&=\Phi_1(f, f-\w f)=:\Psi_1(f,\w f),
\end{aligned}
\end{equation}
and with this new variable, problem \eqref{eq:ABS} is equivalent to
\begin{equation}\label{eq:Psi}
\p_tX=\Psi(X),\qquad X(0)=X_0, 
\end{equation}
where $\Psi:=(\Psi_1,\Psi_2)$ and $X:=(f,\w f)$.
The trivial solution  of \eqref{eq:ABS} corresponds to the solution  $(f,\w f)=(0,0)$ of \eqref{eq:Psi},
so that  we shall study the stability properties of the trivial solution of \eqref{eq:Psi} which, as we shall see, it is more convenient.
This since:

\begin{lemma}\label{L:Le1}
It holds that $\Psi_2(f,\w f)\in h^{1+\alpha}_0(\s)$ for all $(f,\w f)$ in  a zero neighbourhood $\w\W\subset h^{2+2\sign(\gamma_w)+\alpha}(\s)
\times h^{2+2\sign(\gamma_d\cdot\gamma_w)+\alpha}_0(\s)$.
\end{lemma}
\begin{proof} Given $(f, \w f)$ as above,  let again $h=f-\w f$. 
Setting $p:=g\rho_+(1+h)-\gamma_d\kappa(h)$ and $r:=\tr_0\T_2(f)b+g(\rho_+-\rho_-)f+\gamma_w\kappa(f)$, we infer from \eqref{eq:PHI1}  that 
\[
\Phi_1(f,h)=-\cG^{-1}(f,h)\B(f,h)\kS(f,h,p,r),
\]
 which can be reformulated as follows
 \[
\Phi_1(f,h)=\B(f,h)\kS_2(f,h)\tr_0\T_1(f)\Phi_1(f,h)-\B(f,h)\kS(f,h,p,r).
\]
Using this relation  and \eqref{eq:PHI2}, we obtain
\begin{align*}
\int_{\s}\Psi_2(f,\w f)\, dx=&\int_{\s}\B(f,h)\kS(f,h,0,\tr_0\T_1(f)\Phi_1(f,h))\, dx\\[1ex]
&+\int_{\s}\B_1(f,h)\kS(f,h,0,\tr_0\T_1(f)\Phi_1(f,h))\, dx\\[1ex]
&-\int_{\s}\B(f,h)\kS(f,h,p,r) dx-\int_{\s}\B_1(f,h)\kS(f,h,p,r)\, dx.
\end{align*}
Therefore, in order to prove our claim, it will do if we show that 
\[
\int_{\s}\B(f,h)\kS(f,h,p,r) \,dx+\int_{\s}\B_1(f,h)\kS(f,h,p,r)\, dx=0
\]
for all $(f,h)\in\W$ and arbitrary $(p,r)\in (h^{2+\alpha}(\s))^2.$ 
Defining the  harmonic function  $u_+:=\kS(f,h,p,r)\circ\Phi_{f,h}^{-1}$, we then have
\begin{align*}
&\int_{\s}\B(f,h)\kS(f,h,p,r) dx+\int_{\s}\B_1(f,h)\kS(f,h,p,r)\, dx=-\frac{k}{\mu_+}\int_{\G(f)}\p_n u_+\, ds\\[1ex]
&\phantom{space=}-\frac{k}{\mu_+}\int_{\G(f,h)}\p_n u_+\, ds=\frac{k}{\mu_+}\int_{\0(f,h)}\Delta u_+\, d(x,y)=0,
\end{align*} 
where  $n$  stands for the outward unit normal at $\p\0(f,h),$ i.e. $n=\nu$ on $\G(f,h)$ and $n=-\nu$ on $\G(f).$
\end{proof}
\pagebreak

We come now to the proof of our second main result:
\begin{proof}[Proof of Theorem \ref{T:P2}] 
Studying the stability properties  of the trivial solution $(f,h)=(0,0)$ of \eqref{eq:S}, under the constrain \eqref{eq:cond3}, is equivalent to study the stability of the trivial solution
$(f,\w f)=(0,0)$ of \eqref{eq:Psi}, where, by Lemma \ref{L:Le1}, we have that
\[
\Psi:\w\W\to h^{1+\alpha}(\s)\times h^{1+\alpha}_0(\s).
\]
From \eqref{eq:symb} we know that every component of the matrix operator
\[
\p\Psi(0)=
\left[
\begin{array}{cc}
\p_f\Psi_1(0)&\p_{\w f}\Psi_1(0)\\[1ex]
\p_f\Psi_2(0)&\p_{\w f}\Psi_2(0)
\end{array}
\right] 
\]
is a  multiplier with symbol $(\w\lambda^f_1(m))_{m\in\Z}$, $(\w\lambda^{\w f}_1(m))_{m\in\Z\setminus\{0\}},$
$(\w\lambda^f_2(m))_{m\in\Z}$, and $(\w\lambda^{\w f}_2(m))_{m\in\Z\setminus\{0\}}$ respectively, given by
\begin{eqnarray*}
&\w\lambda^f_1(m)=\lambda_1^f(m)+\lambda_1^h(m),&\w\lambda^{\w f}_1(m)=-\lambda_1^h(m),\\[1ex]
&\w \lambda^f_2(m)=\lambda_1^f(m)+\lambda_1^h(m)-\lambda_2^f(m)-\lambda_2^h(m),& \w\lambda^{\w f}_2(m)=-\lambda_1^h(m)+\lambda_2^h(m)
\end{eqnarray*}  
for $m\neq0$,  $\w\lambda^f_1(0)=-kg\rho_-(\mu_++\mu_-)^{-1}$, and $\w\lambda^f_2(0)=0$.
A simple computation shows that the spectrum of the Fr\'echet derivative $\p\Psi(0)$, which coincides with its point spectrum from the same reason \cite[Theorem III.8.29]{K},
is given by
\[
\sigma(\p\Psi(0))=\{\Lambda_\pm(m)\, :\, 1\leq m\in\N\}\cup\{-kg\rho_-(\mu_++\mu_-)^{-1}\}.
\]

Let us start and estimate  the eigenvalues  $\Lambda_\pm(m)$ with $m\in\N.$
Since $[x\mapsto x/\tanh(x)]$ is increasing on $[0,\infty)$ we conclude that
\begin{equation}\label{eq:est1}
\Lambda_-(m)\leq\lambda_1^f(1)/2<0,\quad\forall m\geq1.
\end{equation}
Let now $g(\rho_--\rho_+)+\gamma_w>0.$
 We show that in this case also the other eigenvalues $\Lambda_+(m)$ are of negative sign.
 Indeed, it holds that
 \begin{equation}\label{eq:est2}
\Lambda_+(m)=-4\frac{\lambda_1^f(m)\lambda_2^h(m)-\lambda_1^h(m)\lambda_2^f(m)}{-(\lambda_1^f(m)+\lambda_2^h(m))+\sqrt{(\lambda^f_1(m)-\lambda^h _2(m))^2+4\lambda^h_1(m)\lambda_2^f(m)}},
\end{equation}
and, in view of
\begin{equation}\label{eq:est3}
\lambda_1^f(m)\lambda_2^h(m)-\lambda_1^h(m)\lambda_2^f(m)=\frac{k^2m^2(g\rho_++\gamma_d m^2)[g(\rho_--\rho_+)+\gamma_w m^2]}{\mu_+(\mu_++\mu_-)},
\end{equation}
we conclude that the spectrum of $\p\Psi(0)$ is bounded  away from the negative half axis in $\R^2.$ 
The assertion stated in Theorem \ref{T:P2} $(i)$ follows at once, with the remark that the constant $\omega$ found there may be chosen arbitrarily in the set $(0, -\max\{\lambda\,:\,\lambda\in\sigma(\p\Psi(0)\}).$

On the other hand, if $g(\rho_--\rho_+)+\gamma_w<0,$ one can easily  observe that $\Lambda_+(1)>0.$
Moreover,  since   $\Lambda_\pm(m)\to-\infty$ as $m\to\infty$, we  conclude Theorem \ref{T:P2} $(ii).$
This finishes the proof.
\end{proof}

\section{Finger-shaped equilibria}
This last section is dedicated entirely to the proof of Theorem \ref{T:P3}.
If the tupel $(f, h, u_+,u_-)$ is a stationary solution of \eqref{eq:S}, it must hold that 
$u_-=b,$ $u_+$ is constant, and 
\begin{align}\label{eq:Fi1}
&\gamma_w\kappa(f)+g(\rho_+-\rho_-)f=u_+-b\\[1ex] 
\label{eq:Fi2}
&\gamma_d \kappa(h)-g\rho_+h =b-u_+.
\end{align}
on $\s.$
Equations \eqref{eq:Fi1} and \eqref{eq:Fi2}, called  Laplace-Young or capillarity equations, have  been studied intensively (see \cite{RF} and the literature therein)
subjected to certain constrains at a fixed rigid boundary. 
Though, when dealing with periodic solutions,  we easily get, cf. \cite{EM5}, that $h$ must be constant also in the spatial variable, and if $\rho_+\leq \rho_-$ then also $f$ is constant.
Whence, equations \eqref{eq:Fi1} and \eqref{eq:Fi2} may have nontrivial solutions $(f,h)\notin\R^2$ only when $\gamma_w>0$ and  $\rho_+> \rho_-$.

We are interested to determine only the steady-states $(f,h)$ of \eqref{eq:S} when $b=g\rho_+$ and  which   satisfy \eqref{eq:cond3}, i.e.
$\0(f,h) $ encloses the same volume of fluid as $\0_+.$
Assume by contradiction that $h=c$ for some $c\neq0.$
Since $b-u_+=-g\rho_+c,$ we get that
\[ 
\gamma_w\kappa(f)+g(\rho_+-\rho_-)f=g\rho_+c.
\]
On one hand, if $f$ is constant, it must hold that $f=\rho_+c/(\rho_+-\rho_-),$ which contradicts \eqref{eq:cond3} for $c\neq0.$
On the other hand, the function $p:=f-\rho_+c/(\rho_+-\rho_-)$ solves the equation
\[ 
\gamma_w\frac{p''}{(1+p'^2)^{3/2}}+g(\rho_+-\rho_-)p=0.
\]
The solutions of this equation are, up to a translation, odd. 
Indeed, since $p$ is periodic and nonconstant, it must hold $p(x_0)=0$ for some $x_0\in\R.$
By translation, we may take $x_0=0.$
The pair $(p,q:=p')$ is a global solution of the initial value problem
\begin{equation*}
\left(
\begin{array}{c}
p\\[1ex]
q
\end{array}
\right)'
=
\left(
\begin{array}{c}
q\\[1ex]
-\displaystyle\frac{g(\rho_+-\rho_-)}{\gamma_w} p(1+q^2)^{3/2}
\end{array}
\right),\quad
\left(
\begin{array}{c}
p\\
q
\end{array}
\right)(0)=
\left(
\begin{array}{c}
0\\
p'(0)
\end{array}
\right).
\end{equation*}
This is also true for $(\w p,\w q)(x):=(-p(-x), q(-x))$, $x\in\R.$
Whence, as we claimed, $p$ is odd, so that
$\int_{\s}f\, dx=\rho_+c/(\rho_+-\rho_-)$, which contradicts again $c\neq0$ and \eqref{eq:cond3}.
Consequently, if $(\gamma_w,f,h) $ is a solution of \eqref{eq:S} and \eqref{eq:cond3}, then $h=0$ and $(\gamma_w,f)$ 
solves the problem \eqref{eq:FiF1}, which implies in turn that $f$ has integral mean equal to $0$ and  an  even translation by \cite[Theorem 3.2]{EEM3}.
Combining that particular result with Theorem 6.1 in  \cite{EM12} we obtain all the claims of Theorem \ref{T:P3} excepting the stability assertion.

In the remaining part of this section we prove that the steady-state solution $(\gamma_l(\e), f_l(\e),0)$ of problem \eqref{eq:S} 
is unstable provided that $\e$ is sufficiently small.   
We rediscover first the global branches $(\gamma_l,f_l,0),$ $1\leq l\in\N,$ at least locally near $(\ov\gamma_l,0,0),$
by applying the theorem on bifurcations from simple eigenvalues, due to Crandall and Rabinowitz \cite[Theorem 1.7]{CR}, to the problem
\[
\Psi(\gamma_w,f,\w f)=0,
\]
where $\Psi$ is the mapping defined by \eqref{eq:PSI}.
We shall refer to $(\gamma_w,0,0)$, $\gamma_w>0,$ as being a trivial solution of \eqref{eq:Psi}.
In order to use $\gamma_w$ as a bifurcation argument we establish first analytic dependence of $\Psi$ on $\gamma_w$ 
and take the restriction
\begin{equation}\label{eq:restr}
\Psi:(0,\infty)\times\w \W_{e} \subset h^{4+\alpha}_{e}(\s)\times h^{2+2\sign(\gamma_d)+\alpha}_{0,e}(\s)\to h^{1+\alpha}_{e}(\s)\times h^{1+\alpha}_{0,e}(\s),
\end{equation}
where $\w\W_e:=\w\W\cap (h^{4+\alpha}_{e}(\s)\times h^{2+2\sign(\gamma_d)+\alpha}_{0,e}(\s)).$
That $\Psi$ is well-defined between these spaces follows by using elliptic maximum principles and Lemma \ref{L:Le1}.
When $\gamma_w\notin\{\ov\gamma_l\,:\, 1\leq l\in\N\}$, we infer from \eqref{eq:est1}-\eqref{eq:est3} that all the eigenvalues of $\p_{(f,\w f)}\Psi(\gamma_w,0)$ are different from zero, thus $\p_{(f,\w f)}\Psi(\gamma_w,0)$ is an isomorphism.
The implicit function theorem ensures that $(\gamma_w,0,0) $ is not a bifurcation point
of the trivial solution.
Otherwise, if $\gamma_w=\ov\gamma_l$ for some $l\in\N,$
then 
\begin{align*}
\Kern \p_{(f,\w f)}\Psi(\ov\gamma_l,0)=\spa\{(\cos(lx),\cos(lx))\},
\end{align*}
whereby $(\cos(lx),\cos(lx))$ is the eigenvector corresponding to the eigenvalue $\Lambda_+(l,\ov\gamma_l)=0.$
We considered that $\Lambda_\pm=\Lambda_\pm(m,\gamma_w),$  i.e. $\Lambda_\pm$ depends not only on $m$, but also on $\gamma_w.$ 
Also, the codimension of the image $\im \p_{(f,\w f)}\Psi(\ov\gamma_l,0)$ is one
since 
\[
\left(\frac{a_0}{2}+\sum_{m=1}^\infty a_m\cos(mx),\sum_{m=1}^\infty b_m\cos(mx)\right)\in h^{1+\alpha}_{e}(\s)\times h^{1+\alpha}_{0,e}(\s)
\]
belongs to $\im \p_{(f,\w f)}\Psi(\ov\gamma_l,0)$ if and only if
\[
\left(
a_l, b_l
\right)
=\lambda\left(
1, -\displaystyle\frac{\mu_+(\cosh^2(l)-\cosh(l))+\mu_-(\cosh^2(l)-1)}{\mu_+\cosh(l)}
\right)
\]
for some $\lambda\in\R.$
Moreover, one can easily check that the mixed derivative
\[
\p_{\gamma_w,(f,\w f)}\Psi(\gamma_w,0)(\cos(lx),\cos(lx))=\lambda
\left(
\cos(lx), \frac{1-\cosh(l)}{\cosh(l)}\cos(lx)\right),
\]
whereby $\lambda=-2kl^2\ov\gamma_l\left((\mu_++\mu_-)\sinh(l)\right)^{-1},$  does not belong to the image of $ \p_{(f,\w f)}\Psi(\ov\gamma_l,0)$. 
We conclude by \cite[Theorem 1.7]{CR} the existence of a bifurcation curve
\[
(\gamma_l,f_l,\w f_l):(-\delta_l,\delta_l)\to(0,\infty)\times h^{4+\alpha}_{e}(\s)\times h^{2+2\sign(\gamma_d)+\alpha}_{0,e}(\s)
\]
consisting only of stationary solutions of \eqref{eq:Psi}.
Since they correspond all to a volume of fluid $+$ equal to $2\pi,$ it follows  that $\gamma_l$ and  $\w f_l=f_l$ are, up to a parametrisation, restrictions of the functions obtained in Theorem \ref{T:P3}.

The stability properties of the equilibrium $(\gamma_l(\e), f_l(\e),0) $ for \eqref{eq:S} under the constrain \eqref{eq:cond3}, 
are  equivalent with that of  the steady-state solution $(\gamma_l(\e),f_l(\e),\w f_l(\e))$ of problem \eqref{eq:Psi}.
For our purposes, Theorem \ref{T:P3}, it suffices in fact to show that $(\gamma_l(\e),f_l(\e),\w f_l(\e))$ is an unstable stationary solution of the abstract Cauchy problem
\begin{equation}\label{eq:CPR}
\p_tX=\Psi(\gamma_w,X),\qquad X(0)=X_0,
\end{equation}
where $\Psi$ is the restriction \eqref{eq:restr}.

Indeed, if $\e$ small and $l\geq 2$, then $(\gamma_l(\e),f_l(\e),\w f_l(\e))$  is an unstable solution of \eqref{eq:CPR} since  the eigenvalue 
$\Lambda_+(1,\ov\gamma_l)$ of $\p_{(f,\w f)}\Psi(\ov\gamma_l,0)$ is positive.

For the stability of the stationary  solution $(\gamma_1(\e),f_1(\e),\w f_1(\e))$,  when $ |\e| $ is small and $|\e|\neq0$, it is important how the eigenvalue $\Lambda_+(1,\ov\gamma_1)$ 
(which is  equal to $0$)
perturbs for small $\e.$
Our main tool  is the exchange of stability theorem  \cite[Theorem 1.16]{CR3} due to Crandall and Rabinowitz. 
The assumptions of this theorem are satisfied  by $\Psi$ since: 
\begin{itemize}
\item[$(a)$] $\p_{(f,\w f)}\Psi(\ov\gamma_1,0)$ is a Fredholm operator of index 0 with a one-dimension\-al kernel;
\item[$(b)$] $\p_{\gamma_w,(f,\w f)}\Psi(\ov\gamma_1,0)(\cos(x),\cos(x))\notin\im\p_{(f,\w f)}\Psi(\ov\gamma_1,0)$;
\item[$(c)$] $(\cos(x),\cos(x))\notin\im\p_{(f,\w f)}\Psi(\ov\gamma_1,0).$ 
 \end{itemize}
Letting $\mathcal{J}$ denote the inclusion 
\[
h^{4+\alpha}_{e}(\s)\times h^{2+2\sign(\gamma_d)+\alpha}_{0,e}(\s)\hookrightarrow  h^{1+\alpha}_{e}(\s)\times h^{1+\alpha}_{0,e}(\s),
\]
 in the terminology of \cite{CR3}, $(a)$,  $(b)$, and $(c)$ mean that $0$ is a $\p_{\gamma_w,(f,\w f)}\Psi(\ov\gamma_1,0)$-simple eigenvalue and a $\mathcal{J}$-simple eigenvalue
 of $\p_{(f,\w f)}\Psi(\ov\gamma_1,0)$.
 By choosing $\delta_1$ sufficiently small, we obtain from \cite[Theorem 1.16]{CR3} four continuously differentiable  functions $\lambda:(\ov\gamma_1-\delta_1,\ov\gamma_1+\delta_1)\to\R$, $\mu:(-\delta_1,\delta_1)\to\R$,  $u:(\ov\gamma_1-\delta_1,\ov\gamma_1+\delta_1)\to h^{4+\alpha}_{e}(\s)\times h^{2+2\sign(\gamma_d)+\alpha}_{0,e}(\s),$ and $v:(-\delta_1,\delta_1)\to h^{4+\alpha}_{e}(\s)\times h^{2+2\sign(\gamma_d)+\alpha}_{0,e}(\s)$ such that:
 \begin{align*}
 &\text{$\p_{(f,\w f)}\Psi(\gamma_w,0)u(\gamma_w)=\lambda(\gamma_w)u(\gamma_w)$  for $\gamma_w\in(\ov\gamma_1-\delta_1,\ov\gamma_1+\delta_1)$},\\[1ex]
 &\text{$\p_{(f,\w f)}\Psi(\ov\gamma_1(\e),f_1(\e),\w f_1(\e))w(\e)=\mu(\e)w(\e)$  for $\e\in(-\delta_1,\delta_1)$},\\[1ex]
 &\text{$\lambda(\ov\gamma_1)=\mu(0)=0$, and $u(\ov\gamma_1)=w(0)=(\cos(x),\cos(x)).$}
\end{align*}
Moreover, $\lambda'(\ov\gamma_1)\neq 0$ and 
\[
 \lim_{\e\to0, \mu(\e)\neq0}\frac{-\e\gamma_1'(\e)\lambda'(\ov\gamma_1)}{\mu(\e)}=1.
\]
Since $\lambda(\gamma_w)$ is an eigenvalue of $\p_{(f,\w f)}\Psi(\gamma_w,0)$ and $\lambda(\ov\gamma_1)=0$ we get, by continuity, that $\lambda(\gamma_w)=\Lambda_+(1,\gamma_w)$ for all $|\gamma_w-\ov\gamma_1|<\delta_1.$
Moreover, $\Lambda_+(1,\gamma_w)$ is positive for $\gamma_w<\ov\gamma_1$, and negative  if $\gamma_w>\ov\gamma_1,$ thus $\lambda'(\ov\gamma_1)< 0.$
In order to determine the sign of the eigenvalue $\mu(\e),$ which is the perturbation of the eigenvalue  $0$ of $\p_{(f,\w f)}\Psi(\ov\gamma_1,0)$,
we need to specify the sign of $\gamma_1'(\e)$.
From Theorem \ref{T:P3} we obtain in view of $\gamma_1'(0)=0$ and $\gamma_1''(0)>0$, that $\gamma_1'(\e)$ and $\e$ have the same sign, thus $\mu(\e)$ is a positive eigenvalue, and we are done by
\cite[Theorem 9.1.3]{L}.

\section{Appendix}
We end this paper with a detailed proof of Theorem \ref{T:PHI-O}.
Since the diffeomorphisms used in Section 3 to transform the original problem \eqref{eq:S} into \eqref{eq:TS} are given explicitly in terms of $f$ and $h$, we obtain by direct computation the following
expressions for the elliptic  and trace operators defined  therein: 
\begin{align*}
\A(f)=&\frac{\p^2 }{\p x^2}-2\frac{(1+y)f'}{1+ f}\frac{\p^2 }{\p x\p y}+
\left(\frac{(1+y)^2f'^2}{(1+f)^2}+\frac{1}{(1+f)^2}\right)\frac{\p^2 }{\p y^2}+\\[1ex]
&-(1+y)\frac{(1+f)f''- 2f'^2}{(1+f)^2}\frac{\p }{\p y},\\[1ex]
\A(f,h)=&\frac{\p^2 }{\p x^2}-2\frac{yh'+(1-y)f'}{1+ h-f}\frac{\p^2 }{\p x\p y}+\frac{(yh'+(1-y)f')^2+1}{(1+h-f)^2}\frac{\p^2 }{\p y^2}\\[1ex]
&-\left(\frac{yh''+(1-y)f''}{1+h-f}-2\frac{(h'-f')(yh'+(1-y)f')}{(1+h-f)^2}\right)\frac{\p }{\p y},\\[1ex]
\B(f)=&\frac{k}{\mu_-}\left(\frac{1+f'^2}{1+ f}\tr_0 \frac{\p }{\p y}-f'\tr_0 \frac{\p }{\p x}\right),\\[1ex]
\B(f,h)=&\frac{k}{\mu_+}\left(\frac{1+f'^2}{1+h- f}\tr_0 \frac{\p }{\p y}-f'\tr_0 \frac{\p }{\p x}\right),\\[1ex]
\B_1(f,h)=&-\frac{k}{\mu_+}\left(\frac{1+h'^2}{1+h- f}\tr_1 \frac{\p }{\p y}-h'\tr_1 \frac{\p }{\p x}\right).
\end{align*}
In the following we take $b(0)=c\in\R$ and show that  $\p_f\Phi_i(0)$ and $ \p_h\Phi_i(0),$ $i=1,2,$
are Fourier multiplication operators.
We shall also determine the symbol of these operators which are the main ingredients when proving the Theorems \ref{T:P1} and \ref{T:P2}. 

\subsection{The derivative $\p_f\Phi_1(0)$}
Let us start  by determining the symbol of  the operator $\p_f\Phi_1(0)$.
Putting    $h=0$ in \eqref{eq:PHI1} yields that
\begin{align*}
\Phi_1(0,f,0):=&-\cG^{-1}(f,0)\B(f,0)\kS_2(f,0)\tr_0\T_2(f)c\\[1ex]
&-\cG^{-1}(f,0)\B(f,0)\kS_2(f,0)[g(\rho_+-\rho_-) f+\gamma_{w}\kappa(f)]\\[1ex]
&-\cG^{-1}(f,0)\B(f,0)\kS_1(f,0)g\rho_+.
\end{align*}
We show first that the derivative of $\cF:=\cG^{-1}$ with respect to $f$ in $0$ is a Fourier multiplier.
Indeed, by the chain rule we have, due to $\cF(f,0)\cG(f,0)=\id_{h^{1+\alpha}(\s)},$  that
$\p_f\cF(0,0)[f]M=-\cG^{-1}(0,0)\p_f\cG(0,0)[f]\cG^{-1}(0,0)M$
for all $f\in h^{2+2\sign(\gamma_w)+\alpha}(\s)$ and $M\in\R$.
So, we need to determine 
\begin{equation}\label{eq:P_fG}
\begin{aligned}
\p_f\cG(0,0)[f]M=&-\p_f\B(0,0)[f]\kS_2(0,0)\tr_0\T_1(0)M\\[1ex]
&-\B(0,0)\p_f\kS_2(0,0)[f]\tr_0\T_1(0)M\\[1ex]
&-\B(0,0)\kS_2(0,0)\tr_0\p\T_1(0)[f]M.
\end{aligned}
\end{equation}
Using the expansions found in the proof of Lemma \ref{L:KL}, we find that
$\T_1(0)M=\mu_-k^{-1}M(1+y)$ and $\kS_2(0,0)\tr_0\T_1(0)M=\mu_-k^{-1}M(1-y),$
hence
\begin{equation}\label{eq:P1T2-1}
-\B(0,0)\kS_2(0,0)\tr_0\p\T_1(0)[f]M=\frac{M\mu_-}{\mu_+}f. 
\end{equation}
Concerning the first two terms  of \eqref{eq:P_fG} we determine first an expansion for $\p\T_1(0)[f]M$ and $\p_f\kS_2(0,0)[f]M.$
We start with $\p_f\kS_2(0,0)[f]M.$ 
Elliptic estimates yield  that the function
$\p_f\kS_2(0,0)[f]M$ is the solution of Dirichlet problem
\begin{equation*}
\left\{
\begin{array}{rlllll}
\Delta w&=&-\p_f\A(0,0)[f]\kS_2(0,0)M=-M(1-y)f''&\text{in}& \Omega_+,\\[1ex]
w&=&0&\text{on}& \Gamma_1,\\[1ex]
 w&=&0&\text{on}&\Gamma_{0},
\end{array}
\right.
\end{equation*} 
and, as in the proof of Lemma \ref{L:KL}, we get that
\[
\p_f\kS_2(0,0)[f]M=-\sum_{m\in\Z\setminus\{0\}}M \left(\frac{e^{my}-e^{2m}e^{-my}}{e^{2m}-1}+(1-y)\right)\wh f(m)e^{imx}.
\]
Whence
\begin{equation}\label{eq:P1T2-2}
-\B(0,0)\p_f\kS_2(0,0)[f]\tr_0\T_1(0)M=\frac{M\mu_-}{\mu_+}\sum_{m\in\Z}\frac{m}{\tanh(m)}\wh f(m) e^{imx}-\frac{M\mu_-}{\mu_+}f.
\end{equation}

Consider now the function $\p\T_1(0)[f]M$, which  is the solution of 
\begin{equation*}
\left\{
\begin{array}{rlllll}
\Delta w&=&-\p\A(0)[f]\T_1(0)M=\mu_-k^{-1}M(1+y)f''&\text{in}& \Omega_-,\\[1ex]
\p_yw&=&-\mu_-k^{-1}\p\B(0)[f]\T_1(0)M=\mu_-k^{-1}Mf&\text{on}& \Gamma_0,\\[1ex]
 w&=&0&\text{on}&\Gamma_{-1}.
\end{array}
\right.
\end{equation*}  
Expanding
 \[
\text{$f=\sum_{m\in\Z}\wh f(m) e^{imx }$ and $w(x,y)=\sum_{m\in\Z} w_m(y)e^{imx},$}
\]
we find  that $w_m$ is the solution of the following problem
 \begin{equation*}
\left\{
\begin{array}{rlllll}
 w_m''-m^2w_m&=&-M\mu_-k^{-1}\wh f(m) m^2 (1+y)&\text{in}& -1<y<0,\\[1ex]
w_m'(0)&=&M\mu_-k^{-1} \wh f(m),\\[1ex]
 w_m(-1)&=&0,
\end{array}
\right.
\end{equation*} 
which has the solution $w_m(y)=M\mu_-k^{-1}\wh f(m)(1+y)$ for all $m\in\Z$.
Whence
\[
\p\T_1(0)[f]M=M\mu_-k^{-1}(1+y)f,
\]
and, with the convention that $m/\tanh(m)=1$ if $m=0$, we determine that
\begin{equation}\label{eq:P1T2-3}
-\B(0,0)\kS_2(0,0)\tr_0\p\T_1(0)[f]M=\frac{M\mu_-}{\mu_+}\sum_{m\in\Z}\frac{m}{\tanh (m)}\wh f(m)e^{imx}.
\end{equation}
The relations \eqref{eq:P_fG}-\eqref{eq:P1T2-3} fuse, in view of Lemma \ref{L:KL}, to
\begin{equation}\label{eq:A1}
\p_f\cF(0,0)M=-\frac{2M\mu_-\mu_+}{(\mu_-+\mu_+)^2}\sum_{m\in\Z}\frac{m}{\tanh (m)}\wh f(m)e^{imx}
\end{equation}
for all $f\in h^{2+2\sign(\gamma_w)+\alpha}(\s) $ and $M\in\R.$
In order to determine the derivative $\p_f\Phi_1(0) $  two more steps must be done: we must find the expansions
of the derivatives $\p_f\T_2(0)$ and  $\p_f\kS_1(0,0).$
Since $\T_2(f)M=M$, we get that $\p\T_2(0)[f]M=0,$ and
proceeding similarly as we did before, yields 
\[
\p_f\kS_1(0,0)[f]M=\sum_{m\in\Z\setminus\{0\}}M \left(\frac{e^{my}-e^{2m}e^{-my}}{e^{2m}-1}+(1-y)\right)\wh f(m)e^{imx}.
\]
Combining all these relations, we finally obtain for $\p_f\Phi_1(0,0)$ that

\begin{equation}\label{eq:P1f}
\p_f\Phi_1(0)[f]=\sum_{m\in\Z}\lambda^f_1(m)\wh f(m)e^{imx},
\end{equation}
whereby $(\lambda_1^f(m))_{m\in\Z}$ is given by \eqref{eq:coef_1_f}.

\subsection{The derivative $\p_f\Phi_2(0)$}
The expansions found in the previous subsection are very useful when studying $\p_f\Phi_2(0,0)$, since by \eqref{eq:PHI2}
we have
\begin{equation*}
\begin{aligned}
\Phi_2(0,f,0):=&B_1(f,0)\kS_1(f,0)g\rho_+-B_1(f,0)\kS_2(f,0)\tr_0\T_1(f)\Phi_1(0,f,0)\\[1ex]
&+B_1(f,0)\kS_2(f,0)\tr_0\T_2(f)c\\[1ex]
&+B_1(f,0)\kS_2(f,0)[\gamma_{w}\kappa(f)+g(\rho_+-\rho_-) f].
\end{aligned}
\end{equation*}
In view of $\Phi_1(0)=k(c-g\rho_+)/(\mu_++\mu_-) $ we obtain the following representation
\begin{equation}\label{eq:P2f}
\p_f\Phi_2(0)[f]=\sum_{m\in\Z}\lambda^f_2(m)\wh f(m)e^{imx},
\end{equation}
with symbol $(\lambda^f_2(m))_{m\in\Z}$ given by    \eqref{eq:coef_2_f}.

\subsection{The derivative $\p_h\Phi_1(0)$}
In order to determine a representation of the Fr\'echet derivative $\p_h\Phi_1(0)$ we have to 
investigate first the partial derivatives with respect to $h$ of the solution operators defined in Section 3.
By \eqref{eq:PHI1}, we get for $f=0$  that
\begin{align*}
\Phi_1(0,0,h):=&-\cG^{-1}(0,h)\B(0,h)\kS_2(0,h)\tr_0\T_2(0)c\\[1ex]
&-\cG^{-1}(0,h)\B(0,h)\kS_1(0,h)[g\rho_+(1+h)-\gamma_{d}\kappa(h)].
\end{align*}
Differentiating the relation $\cF(0,h)\cG(0,h)=\id_{h^{1+\alpha}(\s)}$, yields at $h=0$    that
 \[
\p_h\cF(0,0)[h]M=-\cG^{-1}(0,0)\p_h\cG(0,0)[h]\cG^{-1}(0,0)M
\]
for all $M\in\R.$
We infer from the definition of $\cG$ that 
\begin{align*}
\p_h\cG(0,0)[h]=-\p_h\B(0,0)[h]\kS_2(0,0)\tr_0\T_1(0)-\B(0,0)\p_h\kS_2(0,0)[h]\tr_0\T_1(0).
\end{align*}
Given $h\in h^{2+2\sign(\gamma_d)+\alpha}(\s),$ the partial derivative $\p_h\kS_2(0,0)[h]M$ is the solution of the linear Dirichlet problem
\begin{equation*}
\left\{
\begin{array}{rlllll}
\Delta w&=&-\p_h\A(0,0)[h]\kS_2(0,0)M=-Myh''&\text{in}& \Omega_+,\\[1ex]
w&=&0&\text{on}& \Gamma_1,\\[1ex]
 w&=&0&\text{on}&\Gamma_{0},
\end{array}
\right.
\end{equation*}
thus $\p_h\kS_2(0,0)[h]M$ expands as follows
\[
\p_h\kS_2(0,0)[h]M=\sum_{m\in\Z\setminus\{0\}}M\left(\frac{e^{my}-e^{-my}}{e^{m}-e^{-m}}-y\right)\wh h(m)e^{imx},
\]
and similarly
\[
\p_h\kS_1(0,0)[h]M=-\sum_{m\in\Z\setminus\{0\}}M\left(\frac{e^{my}-e^{-my}}{e^{m}-e^{-m}}-y\right)\wh h(m)e^{imx}.
\]
It follows then easily that
\begin{align*}
\p_h\cF(0,0)[h]M=\frac{\mu_-\mu_+M}{(\mu_++\mu_-)^2}\sum_{m\in\Z}\frac{m}{\sinh(m)}\wh h(m)e^{imx}.
\end{align*}
By definition, $\kS_1(0,0)p=(\Delta, \tr,\tr)^{-1}(0,p,0)$ for all $p\in h^{2+\alpha}(\s)$,   
and one can easily check that if $p=\sum_{m\in\Z}\wh p(m)e^{imx}$, then 
\[
\kS_1(0,0)p=y\wh p(0)+\sum_{m\in\Z\setminus\{0\}}\frac{e^{my}-e^{-my}}{e^m-e^{-m}}\wh p(m)e^{imx}.
\]
Summarising, we obtain that
\begin{equation}\label{eq:P1h}
\p_h\Phi_1(0)[h]=\sum_{m\in\Z}\lambda^h_1(m)\wh h(m)e^{imx},
\end{equation}
with symbol $(\lambda_1^h(m))_{m\in\Z}$ given by relation \eqref{eq:coef_1_h}.

\subsection{The derivative $\p_h\Phi_2(0,0)$}
Putting $f=0$ in \eqref{eq:PHI2} yields
\begin{equation*}
\begin{aligned}
\Phi_2(0,h,0)=&B_1(0,h)\kS_1(0,h)[g\rho_+(1+h)-\gamma_{d}\kappa(h)]\\[1ex]
&-B_1(0,h)\kS_2(0,h)\tr_0\T_1(0)\Phi_1(0,h,0)\\[1ex]
&+B_1(0,h)\kS_2(0,h)\tr_0\T_2(0)c.
\end{aligned}
\end{equation*}
Using the relations  derived above, finally yields 
\begin{equation}\label{eq:P2h}
\p_h\Phi_2(0)[h]=\sum_{m\in\Z}\lambda^h_2(m)\wh h(m)e^{imx},
\end{equation}
with symbol $(\lambda_2^h(m))_{m\in\Z}$ given by relation \eqref{eq:coef_2_h}.

\end{document}